\documentclass[a4paper,11pt]{amsart}
\usepackage[utf8]{inputenc}
\usepackage[T1]{fontenc}
\usepackage[english]{babel}
\usepackage{amsmath, amsthm, amssymb, enumerate}
\usepackage{enumerate, enumitem, mathdots, hyperref, array, bm, booktabs, multirow, tabularx, arydshln, mathtools, float}
\theoremstyle{plain}
	\newtheorem{lemma}{Lemma}[section]
	\newtheorem{proposition}[lemma]{Proposition}
	\newtheorem{theorem}[lemma]{Theorem}
	
\theoremstyle{remark}
	\newtheorem{remark}[lemma]{Remark}

\newcommand{\mbF}{\mathbb{F}}
\newcommand{\mbK}{\mathbb{K}}
\newcommand{\mfd}{\mathfrak{d}}
\newcommand{\mfg}{\mathfrak{g}}
\newcommand{\mfn}{\mathfrak{n}}
\newcommand{\mfr}{\mathfrak{r}}
\newcommand{\mfgl}{\mathfrak{g}\mathfrak{l}}
\newcommand{\mfso}{\mathfrak{s}\mathfrak{o}}
\DeclareMathOperator{\der}{Der}
\DeclareMathOperator{\ad}{ad} 
\DeclareMathOperator{\inner}{Inner} 
\DeclareMathOperator{\asoc}{asoc}
\DeclareMathOperator{\im}{Im}
\DeclareMathOperator{\id}{id}
\DeclareMathOperator{\End}{End}
\title{Solvable Quadratic Lie algebras up to dimension 7
}
\author{Pilar Benito} 
\address{Departamento de Matem\'aticas y Computaci\'on, Universidad de La Rioja, 26006 Logro\~no, Spain}
\email{pilar.benito@unirioja.es} 
\thanks{This research has been partially funded by grant Fortalece 2023/03 of ``Comunidad Aut\'onoma de La Rioja'' and by grant  PID2021-123461NB-C21, funded by MCIN/AEI/10.13039/501100011033 and by ``ERDF: A way of making Europe'' starting from 2023. The second author is also by a predoctoral research contract FPI-2024 of ``Universidad de La Rioja''.}

\author{Javier R\'andez-Ib\'a\~nez}
\address{Departamento de Matem\'aticas y Computaci\'on, Universidad de La Rioja, 26006 Logro\~no, Spain}
\email{javier.randez@unirioja.es}
\date{\today}


\begin{document}
    \maketitle

    \begin{abstract}
        The double extension method allows the construction of quadratic Lie algebras from smaller ones. Using this approach, in this work we describe all solvable indecomposable quadratic Lie algebras of dimension at most $7$ over arbitrary fields of characteristic not $2$. In particular, we get complete classifications up to isometric isomorphism over algebraically and real closed fields.
    \end{abstract}
    \section{Introduction}
        Throughout this paper, $\mbF$ denotes a field of characteristic different from $2$, and all vector spaces are assumed to be finite-dimensional over $\mbF$. %
    
        A \textit{Lie algebra} $L$ is a $\mbF$-vector space equipped with a bilinear product $[x,y]$ satisfying \textit{skew-symmetry} and the \textit{Jacobi identity}, $[[x,y],z]+[[y,z],x]+[[z,x],y]=0$. A bilinear form $\varphi$ on $L$ is said to be  \textit{invariant} if $\varphi([x,y],z)=\varphi(x,[y,z])$. A \textit{quadratic Lie algebra} is a pair $(L,\varphi)$, where $L$ is a Lie algebra and $\varphi$ is a nondegenerate, symmetric, invariant bilinear form. Quadratic Lie algebras are also known as metrizable or quasi-classical Lie algebras. If a quadratic algebra $(L,\varphi)$ contains an ideal $I$ such that $\varphi\vert_I$ is non-degenerate, $L=I\oplus I^\perp$ and $L$ is said \textit{decomposable}. Otherwise $L$ is called \textit{indecomposable}. We point out that any quadratic Lie algebra decomposes as an orthogonal sum of indecomposable ideals. Given a Lie algebra $L$, its \textit{derived series} is defined recursively by $L'=[L,L]$ and $L^{(i+1)}=[L^{(i)},L^{(i)}]$. The algebra $L$ is called \textit{solvable} if $L^{(k)}=0$ for some $k\ge 0$. 

        The main goal of this paper is to obtain a complete classification of indecomposable solvable quadratic Lie algebras of dimension at most $7$ over arbitrary fields. The classification of quadratic nilpotent over the complex and real fields up to dimension $\leq 10$ can be found in \cite{Favre_Santharoubane_1987,Kath_2007,Benito_delaConcepcion_Laliena_2017}. And the classification of real nonsolvable quadratic Lie algebras of dimension $\leq 13$, as well as that of solvable quadratic Lie algebras of dimension $\leq 6$, was obtained in \cite{campoamor2008quasi, Benayadi_Elduque_2014}. For complex indecomposable solvable quadratic Lie algebras of dimension at most $8$, we refer to the preprint \cite{duong2014solvable}.

        By the results of Medina and Revoy \cite{Medina_Revoy_1985} and Favre and Santharoubane \cite{Favre_Santharoubane_1987}, every indecomposable solvable quadratic Lie algebra in characteristic zero is either one-dimensional or a double extension of a solvable quadratic Lie algebra of dimension two less. Following \cite{FigueroaOFarrill_Stanciu_1996,Benayadi_Elduque_2014}, in a more general context, the double extension of any quadratic Lie algebra $\mfg_1$ by a Lie algebra $\mfg_2$, not necessarily quadratic, is obtained by combining a central extension with a semidirect product construction involving both algebras. This concept provides a powerful tool to construct examples and to analyze the structure of quadratic Lie algebras. However, the classification of such algebras remains a wild problem, even in the solvable case and in low dimension. Moreover most research on quadratic algebras has been carried out over fields of characteristic zero (see \cite{Ovando_2016} for a survey-guide). In positive characteristic they have not been as extensively studied. 
    
        Along this work, we make use of the notion of a one-dimensional double extension, namely the special case in which $\dim \mfg_2 = 1$. According to \cite[Proposition 2.9]{Favre_Santharoubane_1987} every solvable quadratic Lie algebra can be obtained through a sequence of one-dimensional double extensions.
    
        The paper is organized as follows. Section \ref{general_section} reviews basic properties of quadratic Lie algebras and double extensions, together with the spectral structure of endomorphisms $\delta$ of orthogonal spaces $(V, \varphi)$ and suitable matrix representations of pairs $(\delta,\varphi)$ satisfying $\varphi(\delta(x),y)+\varphi(x, \delta(y))=0$. These tools allow us in Section \ref{section_constructions} to provide explicit bases to describe the Lie brackets and invariant bilinear forms of all indecomposable solvable quadratic Lie algebras of dimension at most $7$. Section \ref{section_alg_real_closed_fields} concludes with their classification over algebraically closed and real closed fields.
    
    \section{General and particular facts}\label{general_section}
        Given a quadratic Lie algebra $(L,\varphi)$ a linear map $\delta\colon L\to L$ is called a \textit{derivation} if $\delta([x,y])=[\delta(x),y]+[x,\delta(y)]$ and it is said \textit{$\varphi$-skew derivation} if $\varphi(\delta(x),y)=-\varphi(x,\delta(y))$. The space of $\varphi$-skew derivations, denoted by $\der_\varphi(L)$, is a subalgebra of the general linear algebra $\mfgl(L)$. For any $x\in L$, the left product by $x$ in $L$, $\ad\, x(y)=[x,y]$, is a $\varphi$-skew derivation. These derivations are called \textit{inner}. Their set, denoted by $\inner(L)$, is a Lie subalgebra of $\der(L)$. So, $\inner(L)\leq \der_\varphi(L)\leq \der (L)\leq \mfgl(L)$.
    
        If we take $\delta\in\der_\varphi(L)$ we can extend the Lie algebra $L$ to the vector space $L_\delta=\mbF\cdot \delta\oplus L\oplus\mbF\cdot \delta^*$ and define the bracket product,\begin{equation}\label{eq_double_product}
            [p \delta+x+q \delta^*,r \delta+y+s\delta^*]_\delta=p \delta(y)-r \delta(x)+[x,y]_L+\varphi(\delta(x),y)\delta^*.          
        \end{equation}
        We can also define the bilinear form,   
        \begin{equation}\label{eq_double_product_form}
            \varphi_\delta(p \delta+x+q \delta^*,r \delta+y+s \delta^*)=ps+qr+\varphi(x,y).
        \end{equation}
        \noindent The pair $(L_\delta,\varphi_\delta)$ is a larger quadratic Lie algebra that will be denoted along the paper by $\mfd(L,\varphi,\delta)$. This construction is known as the \emph{double extension of $L$ by $\delta$}, namely one-dimensional double extension.
    
        For any quadratic Lie algebra $(L,\varphi)$ over a field of  characteristic not $2$, and any ideal $I$ of $L$, the invariance and non-degeneracy of the bilinear form $\varphi$ readily imply the following properties:
        \begin{enumerate}[label=(P\arabic*)] 
            \item \label{item_dim_orthogonal} $I^\perp$ is an ideal of $L$, and $\dim L=\dim I+\dim I^\perp$. Moreover, $L=I\oplus I^\perp$ if $I\cap I^\perp=0$ and hence $\varphi\vert_I$ and $\varphi\vert_{I^\perp}$ are nondegenerate whenever $I,I^\perp\neq 0$. Henceforth, we shall refer to such ideals as {\it regular ideals}.
            \item \label{item_center_derived} $Z(L)^\perp=L^2=[L,L]$ and $(L^2)^\perp=Z(L)$.
            \item \label{item_min_ideal} $I^\perp \subseteq C_L(I)=\{x\in L:[x,I]=0\}$. In particular, if $I$ is a minimal ideal, then $I^\perp$ is a maximal ideal, and either $I$ is a one-dimensional central ideal or $C_L(I)=I^\perp$.
            \item If $L$ is a solvable Lie algebra, $L^2\neq L$ and $0\neq Z(L)=\asoc(L)$, the (direct) sum of all minimal ideals. This follows from \ref{item_center_derived} and \ref{item_min_ideal}.          
        \end{enumerate}

        The next result is well-known in characteristic zero and also holds in arbitrary characteristic according to \cite[Theorem 2.2]{Bordemann_1997}.
        \begin{proposition}\label{prop_double_solvable}
            Any nonabelian solvable quadratic Lie algebra $(L,\varphi)$ over a field of characteristic different from $2$ has a nonzero isotropic central element. Hence, $(L,\varphi)$ arises as a double extension of a solvable quadratic Lie algebra of dimension two less by a $\varphi$-skew derivation.
        \end{proposition}
        \begin{proof}
            Since $Z(L)=(L^2)^\perp$ if there are no central and isotropic elements we have $L=Z(L)\oplus L^2$. Thus $L^2=(L^2)^2$, and $L^2=0$ by solvability. Therefore $L$ is abelian, a contradiction. 
            Now, the proof of Proposition 2.9 in \cite{Favre_Santharoubane_1987} holds in characteristic not $2$, and $L$ can be obtained as a double extension of the $(n-2)$-dimensional quadratic and solvable Lie algebra $\frac{(\mbF\cdot z)^\perp}{\mbF\cdot z}$ for any isotropic central element $z\in L$.
        \end{proof}
        In the sequel, the assumption that the characteristic is different from $2$ will be understood unless otherwise stated. Assume now $(L,\varphi)$ is quadratic and indecomposable. The only indecomposable abelian Lie algebra is the one-dimensional. Suppose $L^2\neq 0$ and note that, by indecomposability:
        
        \begin{enumerate}[label=(P\arabic*)]
            \setcounter{enumi}{4}
            \item \label{item_no_regular_ideals} $L$ has no regular ideals. In particular, minimal ideals are totally isotropic, $I\subseteq I^\perp$. 
            \item \label{item_indecomposable_reduced}$Z(L)\subseteq L^2$: Otherwise, there exists a non-isotropic element $x\in Z(L)\setminus L^2$, so \(\mbF x\) is a regular ideal, contradicting  \ref{item_no_regular_ideals}. Any (arbitrary) Lie algebra satisfying \(Z(L)\subseteq L^2\) is called \emph{reduced}.\footnote{The name comes from \cite{Tsou_Walker_1957}, where it is proved that every quadratic Lie algebra decomposes as an orthogonal direct sum $M=L\oplus A$ where $A$ is abelian quadratic and $L$ is quadratic and reduced.}
        \end{enumerate}
    
        \begin{lemma}\label{lemma_dimension_bounds}
            Let $(L,\varphi)$ be a non-abelian quadratic Lie algebra. Then $\dim L^2\geq 3$. If, in addition, $L$ is solvable and reduced, then $\dim L=\dim L^2+q$ for some $1\leq q\leq \dim L^2$ and $\left\lceil \frac{\dim L}{2} \right\rceil\leq \dim L^2\leq \dim L-1$. 
        \end{lemma}
    
        \begin{proof}
            For $\dim L^2\leq 2$, we have $\dim Z(L)=\dim L-\dim L^2\geq n-2$ and $L=Z(L)+\mbF x+\mbF y$. Hence $L^2=\mbF[x,y]$ and $\dim L^2\leq 1$. Repeating the reasoning, $L=Z(L)+\mbF x$ and $L^2=0$. Therefore $L$ is abelian whenever $\dim L^2\leq 2$. The latter assertion follows immediately from  \ref{item_dim_orthogonal} and \ref{item_indecomposable_reduced}.
        \end{proof}
    
        \begin{lemma}\label{lemma_indecomposable}
            Any solvable quadratic and reduced Lie algebra $L$ of dimension $\leq 7$ is indecomposable. Even more, $\dim L=1,4,5,6, 7$.  
        \end{lemma}
    
        \begin{proof}
            If $L$ is decomposable then $L=I_1\oplus I_2$ with both ideals regular and $I_1^\perp=I_2$. We may assume w.l.o.g. $\dim I_1\leq 3$, thus $I_1^2=0$ by applying Lemma \ref{lemma_dimension_bounds}. But this implies $I_1\subseteq Z(L)\setminus L^2$ contrary to the reduced assumption. This proves the indecomposability of $L$. The last assertion is easily checked from Lemma \ref{lemma_dimension_bounds} 
        \end{proof}
    
        \begin{remark}
            In Section \ref{section_constructions}, we construct indecomposable solvable quadratic Lie algebras of dimensions $4,5,6,7$.
            As a consequence, reduced quadratic Lie algebras exist in every dimension $d\geq 8$  over fields of characteristic different from $2$. This existence result was previously established in characteristic zero in \cite{tsou1962xi}.  
        \end{remark}

        Given $\delta_1,\delta_2\in\der_{\varphi}(L)$, we give a sufficient condition for the corresponding double extensions to be isometrically isomorphic. The result is just \cite[Proposition 2.11]{Favre_Santharoubane_1987} (see also \cite[Lemma 1.4]{duong2014solvable}). We include a straightforward proof sketch showing that the result remains valid over arbitrary fields of characteristic not $2$.
    
        \begin{proposition}\label{prop_iso_iso_reduction_one}
            Let $\lambda\in\mbF^*$, $x\in L$, $\delta\in \der_\varphi(L)$ and $\sigma\colon(L, \varphi)\to (L, \varphi)$ an isometric isomorphism. Then, the map $\delta(\lambda, x,\sigma)=\lambda\sigma^{-1}\delta\sigma+\ad\, x$ is a $\varphi$-skew derivation of $L$ and the double extensions $\mfd(L,\varphi,\delta)$ and $\mfd(L,\varphi,\delta(\lambda, x,\sigma))$ are isometrically isomorphic quadratic Lie algebras.
        \end{proposition}
        \begin{proof}
            One can readily check that $\delta(\lambda,x,\sigma)$ is a $\varphi$-skew derivation and the map $g\colon\mfd(L,\varphi,\delta(\lambda, x)))\to\mfd(L,\varphi,\delta)$ given by 
            \[
                g(\delta(\lambda,x,\sigma))=\lambda\delta+\sigma(x)-\frac{\varphi(x,x)}{2\lambda}\delta^*,\  g(y)=\sigma(y)-\frac{\varphi(x,y)}{\lambda}\delta^* \quad \forall y\in L,
            \]
            and $g(\delta(\lambda,x,\sigma)^*)=\frac{\delta^*}{\lambda}$ is an isometric isomorphism.
        \end{proof}

        By Proposition~\ref{prop_iso_iso_reduction_one}, any double extension by an inner derivation is isometrically isomorphic to the double extension defined by the zero (inner) derivation $\ad_0:=\ad_L0$. But $D=\mfd(L,\varphi,\ad_0)$ satisfies $\ad_0,\,\ad_0^*\in Z(D)$ and $\ad_0\notin D^2$.
        So $D$ is non-reduced and therefore decomposable from \ref{item_indecomposable_reduced}. Consequently, only \textit{outer}  derivations (derivations that are not inner) need to be considered in the classification of indecomposable quadratic Lie algebras.

        We also state a necessary condition for the existence of isomorphisms and isometric isomorphisms, which is particularly useful in low-dimensional cases.

        \begin{proposition}\label{prop_necessary_isomorphism}
            Let $(L_i,\varphi_i)$, $i=1,2$, be quadratic Lie algebras, and let $D_i=\mfd(L_i,\varphi_i, \delta_i)$. If $D_1$ and $D_2$ are reduced and (isometrically) isomorphic, then the Lie algebras $D_1^2/Z(D_1)$ and $D_2^2/Z(D_2)$ are (isometrically) isomorphic. And so are their derivations algebras    $(\der_{\varphi_i}(D_i))$ $\der (D_i)$.
        \end{proposition}
        \begin{proof}
            The result follows immediately by taking into account that $D_i^2/Z(D_i)$ are quadratic Lie algebras and centers and derived subalgebras are invariant ideals under isomorphisms and isometric isomorphisms of Lie algebras.  
        \end{proof}
        Most low-dimensional quadratic Lie algebras arise as double extensions of an abelian quadratic Lie algebra, namely an orthogonal vector space $(V,\varphi)$. In this setting $\der_\varphi(V)=\mfso(V,\varphi)$ and all inner derivations vanish. Thus, the outer derivations are precisely the nonzero elements of $\mfso(V,\varphi)$. Any double extension $\mfd(V,\varphi, \delta)$ determined by an outer derivation $\delta$ satisfies
        \[
            Z(\mfd(V,\varphi, \delta))=\ker \delta\oplus\mbF\cdot \delta^*\text{\ and\ }(\mfd(V,\varphi, \delta))^2=\im \delta\oplus\mbF\cdot \delta^*,\text{\ so}
        \]
        \begin{enumerate}[label=(P\arabic*)]
            \setcounter{enumi}{6}
            \item $\mfd(V,\varphi, \delta)$ is reduced if and only if $\ker \delta \subseteq\im \delta\neq 0$.
        \end{enumerate}

        To simplify our constructions, we use canonical forms of pairs $(\delta, \varphi)$, with $0\neq\delta \in \mfso(V,\varphi)$. A unified treatment of this classical spectral-theoretic problem over arbitrary fields is presented in \cite{caalim2020isometric} and \cite{benito2024skew}. Throughout the paper, we adopt the notation and results in the latter reference. 
    
        Let $(V,\varphi)$ be an orthogonal $\mbF$-vector space, $f: V\to V$ a linear map, and $V=V_{\pi_1}\oplus \cdots\oplus V_{\pi_r}=V_0 \oplus\bigl(\oplus_{\pi_i\neq x}V_{\pi_i}\bigr)$ the \emph{primary decomposition} associated to the irreducible decomposition of the minimal polynomial of $f$, $m_f(x)=\pi_1^{k_1}(x)\cdots\pi_r^{k_r}(x)$, with $\pi_i(x)$ monic, $k_i\geq 1$ and $V_{\pi_i}=\ker \pi_i^{k_i}(f)$. Let $J_n(\lambda)$ denote the Jordan $n$-by-$n$ canonical block and $C(\pi(x))$ the companion matrix of a given monic polynomial $\pi(x)=x^n-a_{n-1}x^{n-1}-\dots-a_1x-a_0$. Therefore,
        \begin{equation}\label{eq_jordan_canonical_blocks}
	       J_n(\lambda):=\left(\begin{smallmatrix}
		        \lambda & 1       &        & 0       \\
		        0       & \lambda & \ddots & 0       \\
		                &         & \ddots & 1       \\
		        0       & 0       &        & \lambda
	       \end{smallmatrix}\right)^t\!\!,\quad 
	       C(\pi(x)):=\left(\begin{smallmatrix}
		        0 &        & 0 & a_0     \\
		        1 & \ddots &   & a_1     \\
		           & \ddots & 0 & \vdots  \\
		        0 &        & 1 & a_{n-1}
	        \end{smallmatrix}\right)
        \end{equation}
        We also fix matrices for symmetric forms of maximal Witt index:
        \begin{equation}\label{eq_isot_zer_vap_Caalim}
		      B_{2n}:=\left(\begin{smallmatrix}
				  &    &         & 1 \\
				  &    & \iddots &   \\
				  & 1 &         &   \\
				1 &    &         &\end{smallmatrix}\right)\text{\ and\ } 
            B_{2n+1}(\pm):=\left(\begin{smallmatrix}
				  &    &         & 1 \\
				  &    & -1&   \\
				  & \iddots &         &   \\
				1 &    &         &\end{smallmatrix}\right).
	    \end{equation}
        The coordinate expressions of these forms are:
        \begin{equation}
            \sum\limits_{\substack{1\leq i<j\leq 2n\\ i+j=2n+1}}2x_ix_j\quad\text{\ and\ }\sum\limits_{\substack{1\leq i<j\leq 2n+1\\ i+j=2n+2}}(-1)^{i+1}2x_ix_j+(-1)^nx_{n+1}^2.
        \end{equation}

        Assume next $(V,\varphi)$ is an orthogonal vector space and $\delta\in \mfso(V, \varphi)$. Following Section 2 in \cite{benito2024skew}, the set of roots (in any splitting field) of $m_\delta(x)$ is symmetric with respect to zero, i.e., $\{\lambda: m_\delta (\lambda)=0\}=-\{\lambda: m_\delta (\lambda)=0\}$ even more:
    
        \begin{proposition}[\cite{benito2024skew}]\label{prop_skew_perp_decomposition}
            Let $(V,\varphi)$ be an orthogonal $\mbF$-vector space, $\delta\in \mfso(V, \varphi)$ and $m_\delta(x)=\pi_1^{k_1}(x)\cdots \pi_r^{k_r}(x)$ with $\pi_i(x)\in\mbF[x]$ distinct irreducibles and associated primary decomposition $\bigoplus_{i=1}^r\ker \pi_i(f)^{k_i}$. Then there is a permutation $\sigma\in S_r$, $\sigma^2=\id$, defined by $\sigma(i)=j$ if and only if $\gcd(\pi_i(x),\pi_j(-x))= \pi_i(x)$. Then $V_{\pi_i}^\perp=\bigoplus_{j\neq\sigma(i)}V_{\pi_j}$, moreover, if $\sigma(i)=i$ then $V_{\pi_i}$ is a regular $f$-invariant subspace and if $j=\sigma(i)\neq i$ then $V_{\pi_i}\oplus V_{\pi_j}$ is a regular $f$-invariant subspace with totally isotropic summands.
        \end{proposition}
        \noindent For any $\delta\in \mfso(V, \varphi)$, the main facts that follows from Proposition \ref{prop_skew_perp_decomposition} are:
        \begin{enumerate}[label=(P\arabic*)]
            \setcounter{enumi}{7}
            \item\label{item_V_decomposition} $V=V_x\perp V_1\perp V_2$ where $V_x, V_1$ and $V_2$ are regular subspaces. 
            \item\label{item_V1} $V_1$  is the orthogonal sum of the primary components $V_{\pi}$ given by irreducible polynomials $\pi(x)\neq x$ such that  $\pi(\lambda)=0$ iff $\pi(-\lambda)=0$.
            \item\label{item_V2} $V_2$ decomposes as orthogonal sum of regular subspaces $W=V_{\pi_1}\oplus V_{\pi_2}$ with two equidimensional and totally isotropic primary components such that $\pi_1(\lambda)=0$ iff $\pi_2(-\lambda)=0$.
        \end{enumerate}

        \noindent The following result decomposes each primary component of the subspace $V_1$ described in \ref{item_V1}, into smaller orthogonal regular cyclic subspaces.
    
        \begin{lemma}\label{lemma_V1_decomposition}
            Let $(V,\varphi)$ be an orthogonal $\mbF$-vector space, $\delta\in \mfso(V, \varphi)$ and $m_\delta(x)=\pi(x)^k$ for some irreducible polynomial $\pi(x)\neq x$. There exists a polynomial $q(x)\in\mbF[x]$ such that $\pi(x)=q(x^2)$ and $\pi(0)=q(0)\neq 0$. Moreover, $V=\bigoplus_{i=1}^r\mbF[x](\delta)(v_i)$ where each summand is a regular $\delta$-invariant subspace.
        \end{lemma}
        \begin{proof}
            We must have $\gcd(\pi(x),\pi(-x))=\pi(x)$, so $\pi(-x)=\lambda \pi(x)$. Since the coefficients of $\pi(-x)$ are just those of $\pi(x)$ or their opposites we must have $\lambda=\pm 1$. Therefore the monomials of $\pi(x)$ have the same parity. If they are all odd, then $\pi(x)$ must have at least degree $3$ since $\pi(x)\neq x$, but then $x\mid p(x)$. Therefore $\pi(x)$ has only even powers and $\pi(x)=q(x^2)$ for some other polynomial $q(x)$, thus $p(0)=q(0)\neq 0$.
        
            By using the theory for finitely generated modules over PID's, $V$ decomposes as $V=\bigoplus_{i=1}^r\mbF[x](\delta)(v_i)$, where the minimum polynomial of every $v_i\in V$ is $m_{v_i}(x)=\pi(x)^{d_i}$ and $k=d_1\geq d_2\geq\cdots\geq d_r\geq 1$. Denote $\mbF[x](\delta)(v_i)$ as $V_i$. If $V_1$ is not a regular subspace, we can take $0\neq w\in V_1\cap V_1^\perp$, so there are polynomials $r_i(x)$ with less degree than $\pi(x)$ or $r_i(x)=0$ such that $w=\sum_{i=0}^{k-1}\pi(\delta)^ir_i(\delta)(v_1)$; if $j$ is the minimum index such that $r_j(x)\neq 0$, we have $0\neq w'=\pi(\delta)^{k-j-1}(w)=\pi(\delta)^{k-1}r_j(\delta)(v_1)\in V_1\cap\im \pi(\delta)^{k-1}$ and $\varphi(w',t(\delta)(v_1))=\varphi(w,\pi(\delta)^{k-j-1}t(\delta)(v_1))=0$ so $w'\in V_1\cap V_1^\perp\cap \im \pi(\delta)^{k-1}$.  Using the same reasoning, $\mbF[x](\delta)(w')\perp V_1$, so we may take $w'=\pi(\delta)^{k-1}(v_1)$. Thus, there must be $v_i$ with $v_i\not\in\ker \pi(\delta)^{k-1}$, equivalently $m_{v_i}(x)=\pi(x)^k$, such that $w'\not\in V_i^\perp$. If $V_i$ is a regular subspace, we have found an indecomposable regular subspace. In other case, there is $r(x)$ with less degree than $\pi(x)$ such that $\varphi(w',r(\delta)(v_i))\neq 0$. In this case $m_{r(\delta)(v_i)}(x)=m_{v_i}(x)$, therefore me may assume that $\varphi(w',v_i)=2$ after changing names and rescaling. Consider now $v=v_1/2+v_i/2$ and take $u=\pi(\delta)^{k-1}(v)$. Then $\varphi(u,v)=1/4(\varphi(\pi(\delta)^{k-1}(v_1),v_i)+\varphi(\pi(\delta)^{k-1}(v_i),v_1))=1\neq 0$ so $\mbF[x](\delta)(\pi(\delta)^{k-1}(v))\not\perp\mbF[x](\delta)(v)$ and thus $\mbF[x](\delta)(v)$ must be a regular subspace and $m_v(x)=\pi(x)^k$. In any case we may find $W=\mbF[x](\delta)(v)$ regular with $m_v(x)=\pi(x)^k$. Now $W^\perp$ is $\delta$-invariant because $W$ is $\delta$-invariant, and by induction in the dimension of $V$ we can conclude the result.
        \end{proof}

        \noindent Finally, according to \cite[Theorem 2.5]{benito2024skew}, we have:
        \begin{enumerate}[label=(P\arabic*)]
            \setcounter{enumi}{10}
            \item\label{item_V_x} The primary component $V_x$ decomposes as an orthogonal sum of regular invariant cyclic subspaces with minimum polynomial $x^{2n+1}$, and regular invariant subspaces $W=W_1\oplus W_2$, where $W_1$ and $W_2$  are totally isotropic cyclic subspaces with minimum polynomial $x^{2n}$.\item\label{item_isotropic_orthogonal} For each nonzero eigenvalue $\lambda\neq 0$ of $\delta$, $-\lambda$ is also an eigenvalue and $V_{x-\lambda}$ and $V_{x+\lambda}$ are totally isotropic. Moreover, $V_{x-\lambda}\oplus V_{x+\lambda}$ decomposes as an orthogonal sum of regular subspaces $W=W_1\oplus W_2$ where $W_1$ and $W_2$ are totally isotropic cyclic subspaces with minimum polynomials $(x-\lambda)^k$ and $(x+\lambda)^k$, respectively.
        \end{enumerate} 
        Combining facts \ref{item_V_decomposition}–\ref{item_V2}, one readily obtains the elementary divisor structures of skew-symmetric endomorphisms $\delta\colon V\to V$ of low-dimensional orthogonal spaces $(V,\varphi)$ satisfying $\ker\delta\subseteq\im \delta \neq 0$. They are listed in Table~\ref{table_divisors_structure}.
        \begin{table}[h]
            \scriptsize
            \centering
            \begin{tabularx}{\textwidth}{>{\centering\arraybackslash}XXX}
                \toprule
                {\bf Dimension of $V$}  & \multicolumn{2}{l}{\bf Elementary divisors of $\delta\in \mfso(V, \varphi)$}  \\ 
                \midrule 
                \multirow{2}{*}{2}	&  $(x-\lambda),(x+\lambda)$	&	$\lambda\in \mbF^*$\\
   	            &  $(x^2-\lambda)$	&	$\lambda\not\in\mbF^2$\\[1mm]
                \hdashline[1pt/1pt]
                &&\\[-2mm]
                3 	&  $x^3$&\\[1mm]
                \hdashline[1pt/1pt]
                &&\\[-2mm]
                \multirow{8}{*}{4}	&  $x^2, x^2$&\\
                &   $(x-\lambda),(x+\lambda), (x-\mu),(x+\mu)$	&$\lambda, \mu\in \mbF^*$\\
                &   $(x-\lambda)^2,(x+\lambda)^2$	&$\lambda\in \mbF^*$\\
                &   $(x-\lambda),(x+\lambda), (x^2-\mu)$	&$\lambda\in \mbF^*, \mu\notin\mbF^2$\\
                &   $(x^2-\lambda),(x^2-\mu)$	&$ \lambda,\mu\notin\mbF^2$\\
                &   $(x^2-\lambda)^2$	&$\lambda\notin\mbF^2$\\
                & $(x^2+\lambda x+\mu), (x^2-\lambda x+\mu)$    & $\lambda\in\mbF^*,\quad \lambda^2-4\mu\not\in\mbF^2$\\
                & $x^4-\lambda x^2+\mu$	& $\mu,\lambda\pm2\sqrt{\mu}, \lambda^2-4\mu\notin\mbF^2$\\[1mm]
                \hdashline[1pt/1pt]
                &&\\[-2mm]
                \multirow{3}{*}{5} &  $x^5$&\\
                & $x^3, (x-\lambda),(x+\lambda)$	&$\lambda\notin\mbF^*$\\
                & $x^3, (x^2-\lambda)$	&$\lambda\notin\mbF^2$\\
                \bottomrule
            \end{tabularx}
            \caption{Elementary divisors of skew maps of $(V,\varphi)$.}
            \label{table_divisors_structure}
        \end{table}

        Since $\dim(V)\le 5$ suffices to obtain all quadratic Lie algebras of dimension at most $7$, Propositions~\ref{prop_skew_perp_decomposition} and Lemma~\ref{lemma_V1_decomposition} yield suitable bases of $V$ in which any pair $(\delta,\varphi)$, with at most two elementary divisors, is represented by one of the matrix pairs $(M(\delta), M(\varphi))$ listed in  Table~\ref{table_matrices}. The matrix pairs involve the matrices $J_n(\lambda)$, $C(\pi(x))$, $B_{2n}$ and $B_{2n+1}(\pm)$ described in \eqref{eq_jordan_canonical_blocks} and  \eqref{eq_isot_zer_vap_Caalim} and the $4$-parameter regular matrices $M(\lambda, \mu)(\gamma, \nu)$,    \begin{equation}\label{eq_4_parametric_matrix}
            M_{\lambda, \mu}(\gamma, \nu):=\left(\begin{smallmatrix}
                \gamma&0&-\nu&0\\0&\nu&0&\gamma\mu+\nu\lambda\\ -\nu&0&-\gamma\mu-\nu\lambda&0\\ 0&\gamma\mu+\nu\lambda&0&\gamma\lambda\mu+\nu\lambda^2-\nu\mu
            \end{smallmatrix}\right).
        \end{equation}
        The information in the Table~\ref{table_matrices} for orthogonal subspaces of dimension at most $5$ follows from Theorem~2.5 in \cite{benito2024skew} for $V_x$, and for $V_1$ and $V_2$ whenever $m_\delta(x)$ splits over the ground field. The remaining two cases are covered by \ref{item_V_x} and \ref{item_isotropic_orthogonal} and the following lemma.

        \begin{lemma}\label{lemma_table2_proof}
            Let $(V,\varphi)$ be a $4$-dimensional orthogonal space, and let $\delta\in\mfso(V,\varphi)$ such that its minimum polynomial, $m_\delta(x)$, does not split linearly over $\mbF$. Then $m_\delta(x)\in \{(x^2-\lambda),(x^2-\lambda)(x^2-\mu),(x^2-\lambda)^2, (x^2+\lambda x+\mu)(x^2-\lambda x+\mu), x^4-\lambda x^2+\mu\}$ and the attached pairs $(M(\delta), M(\varphi))$ are those included in Table~\ref{table_matrices}.       
        \end{lemma}
        \begin{proof}
            For $x^4-\lambda x^2+\mu$ take a basis with coordinate matrix $M(\delta)=\ $ $C(\pi_{\lambda,\mu}(x))$, and just apply that $\delta$ is $\varphi$-skew. For $x^2-\lambda$ and $(x^2-\lambda)(x^2-\mu)$ the same process works. For $m_\delta(x)=\pi_1(x)\pi_2(x)$ with $\pi_1(x)=x^2+\lambda x+\mu$ and $\pi_2(x)=x^2-\lambda x+\mu$, set $V=V_{\pi_1}\oplus V_{\pi_2}$. Following \cite{benito2024skew}, these subspaces are totally isotropic. In $V_{\pi_1}$ choose a basis $\{v_1,v_2\}$ with coordinate matrix $C(\pi_1(x))$. Since $V$ is nondegenerate, this basis can be extended to a basis $\{v_1,v_2,v_3,v_4\}$ such that $v_3,v_4\in V_{\pi_2}$ and $M(\varphi)=B_4$. The fact $\delta$ is $\varphi$-skew allows us to determine the remaining entries of $M(\delta)$. And for $(x^2-\lambda)^2$ with $\lambda\not\in\mbF^2$, we extend scalars to $\mbK=\mbF(\sqrt{\lambda})\cong\mbF\oplus\mbF\sqrt{\lambda}$. In $V_\mbK\cong V\oplus \sqrt{\lambda}V$,  choose a basis $\{v_1+\sqrt{\lambda}v_2,v_3+\sqrt{\lambda}v_4,v_1-\sqrt{\lambda}v_2,v_3-\sqrt{\lambda}v_4,\}$,  $v_1,v_2,v_3,v_4\in V$ with $M_{\mbK}(\delta)=\left(\begin{smallmatrix}
            J_2(\sqrt{\lambda})&0\\0&J_2(-\sqrt{\lambda})
            \end{smallmatrix}\right)$ and $M_{\mbK}(\varphi)=\left(\begin{smallmatrix}
            0&0&\beta&\alpha\sqrt{\lambda}\\
            0&0&-\alpha\sqrt{\lambda}&0\\
            \beta&-\alpha\sqrt{\lambda}&0&0\\
            \alpha\sqrt{\lambda}&0&0&0
            \end{smallmatrix}\right)$,  $\alpha,\beta\in\mbF$. From the basis $\{v_1,v_2,v_3,v_4\}$, we get $M(\delta)=\left(\begin{smallmatrix}
            0&1&0&0\\
            \lambda&0&0&0\\
            1&0&0&1\\
            0&1&\lambda&0
            \end{smallmatrix}\right)$ and $M(\varphi)=\left(\begin{smallmatrix}
            \beta/2&0&0&\alpha/2\\
            0&-\beta/2\lambda&-\alpha/2&0\\
            0&-\alpha/2&0&0\\
            \alpha/2&0&0&0
            \end{smallmatrix}
            \right)$. Setting now  $v_1'=v_1+\beta/2\alpha v_4$, $v_2'=v_2+\ $ $\beta/2\alpha\lambda v_3$, $v_3'=v_3$ and $v_4'=-v_4$ we obtain the matrices in Table~\ref{table_matrices}.
        \end{proof}
        \begin{table}[h]
            \scriptsize
            \centering
            \begin{tabularx}{\textwidth}{>{\hsize=.2\hsize}XX>{\centering\arraybackslash}X>{\hsize=.37\hsize\centering\arraybackslash}X}
                \toprule
                $V$ &Elementary Divisors & $M(\delta)$ & $M(\varphi)$\\
                \midrule
                \multirow{3.5}{*}{$V_x$}&$x^3$ & $J_3(0)$ &$\gamma B_3(\pm)$ \\[1mm]
                &$x^2, x^2$ & \scalebox{.8}{$\begin{pmatrix}
                    J_2(0)&0\\
                    0&-J_2(0)
                \end{pmatrix}$} & $B_4$\\[1mm]
                &$x^5$ & $J_5(0)$ &$\gamma B_5(\pm)$ \\[1mm]
                \hdashline[1pt/1pt]
                &&&\\[-2mm]
                \multirow{6}{*}{$V_1$}&$(x^2-\lambda), \lambda\not\in\mbF^2$ & \scalebox{.8}{$\begin{pmatrix}
                    0&1\\\lambda&0
                \end{pmatrix}$ }&\scalebox{.8}{$\gamma\begin{pmatrix}
                    -\lambda&0\\0&1
                \end{pmatrix}$}\\[2mm]
                &$(x^2-\lambda)^2, \lambda\not\in\mbF^2$ & \scalebox{.8}{$\begin{pmatrix}
                    0&1&0&0\\\lambda&0&0&0\\1&0&0&-1\\0&-1&-\lambda&0
                \end{pmatrix}$} &$\gamma B_4$\\
                &$\displaystyle \pi_{\lambda,\mu}(x)=x^4-\lambda x^2+\mu$ & \multirow{2}{*}{$\displaystyle C(\pi_{\lambda,\mu}(x))$} &\multirow{2}{*}{$\displaystyle M_{\lambda, \mu}(\gamma, \nu)$}\\
                &$\mu,\lambda\pm2\sqrt{\mu},\lambda^2-4\mu\not\in\mbF^2$& &\\[1mm]
                \hdashline[1pt/1pt]
                &&\\[-2mm]
                \multirow{6}{*}{$V_2$}&$(x-\lambda), (x+\lambda), \lambda\in\mbF^*$ & \scalebox{.8}{$\begin{pmatrix}
                \lambda&0\\0&-\lambda
                \end{pmatrix}$} &$B_2$\\[2mm]
                &$(x-\lambda)^2, (x+\lambda)^2,\lambda\in\mbF^*$ & \scalebox{.8}{$\begin{pmatrix}
                    J_2(\lambda)&0\\
                    0&-J_2(\lambda)
                \end{pmatrix}$} & $B_4$\\[1mm]
                &&\multirow{4}{*}{\scalebox{.8}{$\begin{pmatrix}
                    0&-\mu&0&0\\1&-\lambda&0&0\\0&0&\lambda&\mu\\0&0&-1&0
                \end{pmatrix}$}}&\\
                &$\displaystyle(x^2+\lambda x+\mu), (x^2-\lambda x+\mu)$ &  &\multirow{2}{*}{$B_4$}\\
                &$\lambda\neq 0,  \lambda^2-4\mu\not\in\mbF^2$&&\\
                &&&\\\bottomrule
            \end{tabularx}
            \caption{Matrices corresponding to elementary divisors.}
            \label{table_matrices}
        \end{table}

        The final lemma of this section will be useful in the next.
        \begin{lemma}\label{lemma_der_inner}
            Let $(L,\varphi_\delta)=\mfd(V,\varphi, \delta)$ be a one-dimensional double extension of a $2$-dimensional abelian quadratic algebra $V$. Then $\der L=\mfso(L,\varphi_\delta)$ if $\delta=0$, whereas $\der_{\varphi_\delta}L=\inner L$ if $\delta\neq 0$.
        \end{lemma}
        \begin{proof}
            Assume $\delta\neq 0$, otherwise, $L$ is abelian and the result is immediate. We point out that $\dim\mfso(V,\varphi)=1$, thus $\mfso(V,\varphi)=\mbF\cdot\delta$. The possible structures for the elementary divisors of $\delta$ appear in Table~\ref{table_divisors_structure}, thus $\delta$ is a $\varphi$-skew isomorphism. Therefore, $Z(L)=\mbF\cdot \delta^*$ and $L^2=V\oplus\mbF\cdot \delta^*$. Since both ideals are invariant by derivations, any $d\in \der_{\varphi_\delta} L$ is given explicitly as ($v$ is an arbitrary element of $V$)
            \[
                d(\delta)=a\delta+v_0+b\delta^*,\  d(v)=g(v)+h(v)\delta^* \text{\ and \ } d(\delta^*)=c\delta^*,
            \]
            for some $a,b,c\in\mbF$, $v_0\in V$, $g\in\End\,(V)$ and $h\in V^*$.
            And using that $\varphi_\delta(d(v), w)=\varphi_\delta(v, d(w))$, we get $b=0,\, a=-c$, $h(v)=-\varphi(x,v)$ and $\varphi(g(v),w)=-\varphi(v,g(w))$, thus $g\in \mfso(V,\varphi)$. Hence  $g=\lambda \delta$ for some $\lambda \in \mbF$. Since $\varphi$ is non-degenerate, there are $v,w\in V$ such that $[v,w]=\varphi(\delta(v),w)\delta^*=\delta^*$. Applying now that $d$ is a derivation,
            \begin{multline*}
                c\delta^*=d([v,w])=[d(v),w]+[v,d(w)]=[g(v),w]+[v,g(w)]\\=\varphi(\delta(g(v)),w)\delta^*+\varphi(\delta(v),g(w))\delta^*\\=\lambda(\varphi(\delta^2(v),w)-\varphi(\delta^2(v),w))\delta^*=0.
            \end{multline*}
            So $c=0=a$ and $\dim\der_\varphi(L)\leq 3$.  But $\dim\inner(L)=\dim L/Z(L)=3$ and $\inner(L)\subseteq \der_\varphi(L)$, therefore both vector spaces are equals.
        \end{proof}
    
    \section{Construction of Algebras}\label{section_constructions}
        Throughout this section, we construct indecomposable quadratic Lie algebras $\mfd(L,\varphi,\delta)$ via the double extension process. Hence $\delta\in\der_\varphi(L)$ and $\delta$ outer by Proposition \ref{prop_iso_iso_reduction_one}. The Lie bracket is determined by \eqref{eq_double_product}, and the invariant form is obtained from \eqref{eq_double_product_form}. More precisely, if $M(\varphi)$ denotes the matrix of $\varphi$ (see Table~\ref{table_matrices}), then
        \begin{equation}
            \small   \widehat{M(\varphi)}:=M(\varphi_\delta)=\left( \begin{array}{c|c|c} 0 & \bm{0} & 1 \\ \hline \bm{0} & M(\varphi) & \bm{0} \\ \hline 1 & \bm{0} & 0 \\\end{array} \right).
        \end{equation} 
        We also use the matrix forms $B_{2n}$, $B_{2n+1}(\pm)$ and $M(\lambda,\mu)(\gamma, \nu)$ introduced earlier (see \eqref{eq_isot_zer_vap_Caalim} and \eqref{eq_4_parametric_matrix}). We point out that all algebras obtained by iterated double extension in this section have totally isotropic center and are therefore reduced by \ref{item_center_derived}. It then follows from Lemma~\ref{lemma_indecomposable} that all algebras of dimension at most $7$ described here are indecomposable.
    
        \subsection{Algebras up to dimension 5} 
            The unique indecomposable quadratic abelian Lie algebra is one-dimensional. 
            \begin{itemize}[leftmargin=*]
                \item $\bm{(\mfr_1,\varphi_1^\gamma):}$ One-parameter family given by $\gamma \in \mbF^*$; the algebra has a  basis $\{x\}$ with $[x,x]=0$ and $\varphi_1^\gamma(x,x)=\gamma$.
            \end{itemize}
            
            By Lemma~\ref{lemma_dimension_bounds}, every non-abelian solvable reduced quadratic Lie algebra has dimension at least~$4$.  Thus, we begin the construction of non-abelian algebras with a $2$-dimensional orthogonal vector space $(V,\varphi)$ and a nonzero element $\delta\in\mfso(V,\varphi)$.
    
            Using Table \ref{table_divisors_structure} and rescaling $\delta$ if necessary, we may assume that the elementary divisors of $\delta$ are either $(x-1)$, $(x+1)$ or $(x^2-\lambda)$, $\lambda\not\in\mbF^2$. Hence,  the one-dimensional quadratic algebra and the above elementary divisor structures give rise to the following $4$-dimensional quadratic Lie algebras:
 
            \begin{itemize}[leftmargin=*]
                \item $\bm{(\mfr_{4,1},\varphi_{4,1}):}$ The algebra has a basis $\{a,x,y,z\}$ with nonzero products $[a,x]=x,\, [a,y]=-y,\, [x,y]=z$ and $B_4$ as matrix of $\varphi_{4,1}$.
                \item $\bm{(\mfr_{4,2}^\lambda,\varphi_{4,2}^{\gamma,\lambda}):}$ Two-parameter family $\lambda \notin \mbF^2$ and $\gamma \in \mbF^*$; the algebra has a basis $\{a,x,y,z\}$ with nonzero products
                $[a,x]=\lambda y,\, [a,y]=x,\, [x,y]=\lambda z$ and \scalebox{0.7}{$\gamma\widehat{\begin{pmatrix}  -\lambda&0\\0&1
                \end{pmatrix}}$} as matrix of $\varphi_{4,2}^{\gamma,\lambda}$.
            \end{itemize} 

            Let now $w\in \mfr_{4,1}$. Since $\mfr_{4,1}^2/Z(\mfr_{4,1})=\mbK\cdot\bar{x}\oplus\mbK\cdot\bar{y}$ is a $2$-dimensional abelian algebra, the induced map $\overline{\ad\,w}=\overline{\ad\,(ka)}$ has minimal polynomial $(x-k)(x+k)$, $k\neq 0$. Similarly, if $w\in \mfr_{4,2}^\lambda$, the induced map $\overline{\ad\,w}$ has minimal polynomial $(x^2-\lambda k^2)$ where $\lambda k^2\not\in \mbF^2$. Therefore, $\mfr_{4,1}$ and $\mfr_{4,2}^\lambda$ are not isomorphic, by a direct application of Proposition~\ref{prop_necessary_isomorphism}.

            The next step is $\dim V=3$. In this case $x^3$ is the unique elementary divisor according to Table \ref{table_divisors_structure}. This yields to the double extension:
            \begin{itemize}[leftmargin=*]
                \item $\bm{(\mfr_5,\varphi_5^\gamma):}$ One-parameter family given by $\gamma\in \mbF^*$; the algebra has a basis $\{x_1,x_2,y,z_1,z_2\}$ with nonzero products $[x_1,x_2]=y,\, [x_1,y]=z_1,\, [x_2,y]=z_2$ and $\gamma B_5(\pm)$ as matrix of $\varphi_5^\gamma$.
            \end{itemize}
    
            \noindent The only quadratic nilpotent and non-abelian algebra is $\mfr_5$. It is the well known $3$-step free nilpotent Lie algebra $\mfn_{2,3}$ of type $2$ (see \cite{del2012free}).
        \subsection{Algebras of dimension 6}
            These algebras arise as double extensions of a $4$-dimensional solvable quadratic Lie algebra $(L,\varphi)$. From Lemma~\ref{lemma_der_inner}, $\der_\varphi(L)=\inner (L)$ if $L=\mfr_{4,1}, \mfr_{4,2}^\lambda$. Hence, all their double extensions are decomposable. Therefore, indecomposable $6$-dimensional quadratic algebras arise only from $4$-dimensional orthogonal vector spaces $(V,\varphi)$.

            By Table~\ref{table_divisors_structure}, we obtain eight families. Some parameters may be normalized to $1$ by rescaling $\delta$. Using the information provided in Table~\ref{table_matrices}, we can explicitly describe both the Lie bracket and the invariant bilinear form taking into account that $V=V_x\perp V_1\perp V_2$ and $4=\dim V_x+\dim V_1+\dim V_2$. The following list exhausts all possibilities.

            \begin{itemize}[leftmargin=*]
                \item $\bm{(\mfr_{6,1},\varphi_{6,1}):}$ Elementary divisors $x^2,x^2$. The algebra has a basis $\{a,x_1,x_2,$ $y_1,y_2,z\}$ with nonzero products $[a,x_1]=x_2,\, [a,y_1]=-y_2,\, [x_1,y_1]=z,$ and $B_6$ as matrix of $\varphi_{6,1}$. 
                \item $\bm{(\mfr_{6,2}^\lambda,\varphi_{6,2}):}$ Elementary divisors $(x-1), (x+1)$ and $(x-\lambda),(x+\lambda)$. one-Parameter family with $\lambda\in \mbF^*$. The algebra has a basis $\{a,x_1,x_2,y_1,y_2,z\}$ with nonzero products $[a,x_1]=x_1, [a,x_2]=-x_2,\, [a,y_1]=\lambda y_1,$, $[a,y_2]=-\lambda y_2,\, [x_1,x_2]=z, [y_1,y_2]=\lambda z$ and \scalebox{0.7}{$\widehat{\begin{pmatrix}B_2&\bm{0}\\\bm{0}&B_2\end{pmatrix}}$} as matrix of $\varphi_{6,2}$.  
                \item $\bm{(\mfr_{6,3},\varphi_{6,3}):}$ Elementary divisors $(x-1)^2, (x+1)^2$. The algebra has a basis $\{a,x_1,x_2,$ $y_1,y_2,z\}$ with nonzero products $[a,x_1]=x_1+x_2, [a,x_2]=x_2,[a,y_1]=-y_1-y_2,$ $[a, y_2]=-y_2,[x_1,y_1]=[x_1,y_2]=[x_2,y_1]=z$ and $B_6$ as matrix of $\varphi_{6,3}$. 
                \item $\bm{(\mfr_{6,4}^\lambda,\varphi_{6,4}^{\gamma,\lambda}):}$ Elementary divisors $(x-1), (x+1)$, $(x^2-\lambda)$. Two-parameter family, $\lambda\not\in\mbF^2$ and $\gamma\in \mbF^*$. The algebra has a basis $\{a,x_1,x_2,$ $y_1,y_2,z\}$ with nonzero products $[a,x_1]=x_1, [a,x_2]=-x_2, [a,y_1]=\lambda y_2$, $[a,y_2]=y_1, [x_1,x_2]=z, [y_1,y_2]=\lambda z$ and \scalebox{0.7}{$\gamma\widehat{\begin{pmatrix}B_2&\bm{0}\\\bm{0}&A_\lambda\end{pmatrix}}$}, $A_\lambda=$\scalebox{0.7}{$\begin{pmatrix}  -\lambda&0\\0&1
                \end{pmatrix}$}, as matrix of $\varphi_{6,4}^{\gamma,\lambda}$.
                \item $\bm{(\mfr_{6,5}^{\lambda, \mu,\gamma},\varphi_{6,5}^{\lambda, \mu,\gamma,\nu}):}$ Elementary divisors $(x^2-\lambda), (x^2-\mu)$ $\lambda, \mu\not\in\mbF^2$. Four-parameter family, $\lambda, \mu\not\in\mbF^2$ and $\gamma, \nu\in \mbF^*$. It has a basis $\{a,x_1,x_2, y_1,y_2,$ $z\}$ with nonzero products $[a,x_1]=\lambda x_2, [a,x_2]=x_1, [a,y_1]=\mu y_2$, $[a,y_2]=y_1, [x_1,x_2]=\lambda z, [y_1,y_2]=\gamma\mu z$ and \scalebox{0.7}{$\nu\widehat{\begin{pmatrix}A_\lambda&\bm{0}\\\bm{0}&\gamma A_\mu\end{pmatrix}}$} as matrix of $\varphi_{6,5}^{\lambda, \mu,\gamma,\nu}$.
                \item $\bm{(\mfr_{6,6}^\lambda,\varphi_{6,6}^\gamma):}$ Elementary divisor $(x^2-\lambda)^2$. Two-parameter family, $\lambda\not\in\mbF^2$ and $\gamma\in \mbF^*$. The algebra has a basis $\{a,x_1,y_1,x_2,y_2,z\}$ with nonzero products $[a,x_1]=\lambda y_1+x_2, [a,y_1]=x_1-y_2, [a,x_2]=-\lambda y_2$, $[a,y_2]=-x_2$, $[x_1,y_1]= [y_1,y_2]=z, [x_1,x_2]=\lambda z$ and $\gamma B_6$ as matrix of $\varphi_{6,6}^\gamma$.
                \item $\bm{(\mfr_{6,7}^\lambda,\varphi_{6,7}):}$ Elementary divisors $(x^2+2 x+\lambda), (x^2-2 x+\lambda)$. One-parameter family, $1-\lambda\not\in\mbF^2$. It has a basis $\{a,x_1,y_1,$ $x_2,y_2,z\}$ with nonzero products $[a,x_1]=y_1$, $[a,y_1]=-\lambda x_1+2 y_1$, $[a,x_2]=-2 x_2-y_2$, $[a,y_2]=\lambda x_2$, $[x_1,x_2]=z$, $[y_1,x_2]=2 z$, $[y_1,y_2]=-\lambda z$ and $B_6$ as matrix of $\varphi_{6,7}$.
                \item $\bm{(\mfr_{6,8}^{\lambda,\mu,\gamma},\varphi_{6,8}^{\lambda,\mu,\gamma, \nu}):}$ Elementary divisor $x^4-\lambda x^2+\mu$. Four-parameter family, $\mu,\lambda\pm 2 \sqrt{\mu}, \lambda^2-4\mu\notin\mbF^2$. The algebra has a basis $\{a,x_1,x_2,x_3,x_4,z\}$ with nonzero products $[a,x_i]=x_{i+1},i=1,2,3$, $[a,x_4]=-\mu x_1+\lambda x_3$, $[x_1,x_2]=z$, $[x_1,x_4]=-[x_2,x_3]=(\gamma\mu+\lambda)z$, $[x_3,x_4]=(\gamma\lambda\mu+\lambda^2-\mu)z$ and $\nu\widehat{M_{\lambda,\mu}(\gamma,1)}$ as matrix of $\varphi_{6,8}^{\lambda,\mu,\gamma, \nu}$.
            \end{itemize}  
            For each $\mfr_{6,j}$, every $w\in \mfr_{6,j}\backslash (\mfr_{6,j})^2$ induces a derivation $\overline{\ad\,w}$ of $\mfr_{6,j}^2/Z(\mfr_{6,j})$ whose elementary divisors are the same as those defining the algebra. Hence, Proposition~\ref{prop_necessary_isomorphism} implies that these families are pairwise non-isomorphic.

            The algebra $\mfr_{6,1}$ is the only quadratic $6$-dimensional nilpotent algebra.  It is the $2$-step free nilpotent algebra $\mfn_{3,2}$ of type $3$. According to \cite{del2012free}, the Lie algebras $\mfn_{2,3}$ and $\mfn_{3,2}$ are the unique free nilpotent quadratic Lie algebras.

        \subsection{Algebras of dimension 7}
            The algebras in this section arise as double extensions of one of the following $5$-dimensional quadratic algebras: abelian, $\mfr_5$, $\mfr_{4,1}\oplus\mfr_1$ and $\mfr_{4,2}^\lambda\oplus\mfr_1$ (here $\mfr_1:=\mbF c$). We start with the double extensions of a $5$-dimensional orthogonal space $(V,\varphi)$. Once again, Tables~\ref{table_divisors_structure} and~\ref{table_matrices} are used to identify the (indecomposable) extensions.

            \begin{itemize}[leftmargin=*]
                \item $\bm{(\mfr_{7,1},\varphi_{7,1}^\gamma):}$ Elementary divisor $x^5$. One-parameter family, $\gamma \in \mbF^*$. The algebra has a basis $\{a,x_1,x_2,x_3,x_4,x_5,z\}$ with nonzero products $[a,x_i]=x_{i+1}, i=1,2,3,4$,  $[x_1,x_4]=-[x_2,x_3]=z$ and $\gamma B_7(\pm)$ as matrix of $\varphi_{7,1}$.
                \item $\bm{(\mfr_{7,2},\varphi_{7,2}^\gamma):}$ Elementary divisors $(x-1), (x+1)$ and $x^3$. One-Parameter family, $\gamma \in \mbF^*$. The algebra has a basis $\{a,x_1,x_2,y_1,y_2,y_3,z\}$ with nonzero products $[a,x_1]=x_1, [a,x_2]=-x_2$, $[a,y_i]=y_{i+1},i=1,2$, $[x_1,x_2]=z, [y_1,y_2]=z$ and \scalebox{0.7}{$\gamma\widehat{\begin{pmatrix}B_2&\bm{0}\\\bm{0}&-B_3(\pm)\end{pmatrix}}$} as matrix of $\varphi_{7,2}^\gamma$. 
                \item $\bm{(\mfr_{7,3}^{\lambda,\gamma},\varphi_{7,3}^{\lambda,\gamma,\nu}):}$ Elementary divisors $(x^2-\lambda)$ and $x^3$. Three-parameter family $\lambda\not\in\mbF^2, \gamma,\nu\in \mbF^*$ The algebra has a basis $\{a,x_1,x_2,y_1,y_2,y_3,z\}$ with nonzero products $[a,x_1]=\lambda x_2$, $[a,x_2]=x_1$, $[a,y_i]=y_{i+1}, i=1,2$, $[x_1,x_2]=\lambda\gamma z$, $[y_1,y_2]=z$ and \scalebox{0.7}{$\nu\widehat{\begin{pmatrix}\gamma A_\lambda&\bm{0}\\\bm{0}&-B_3(\pm)\end{pmatrix}}$} as matrix of $\varphi_{7,3}^{\lambda,\gamma, \nu}$.
            \end{itemize}
    
            As $\mfr_5=\mfn_{2,3}$, every skew-derivation is completely determined by its action on the generators ($a$ and $x_1$ in the chosen basis of $\mfr_5$). Exploiting this fact, a straightforward computation yields the following matrix forms (see \cite[Section~4]{del2012free} for a description of $\der_{\varphi_5}\mfr_5$ and also for the isometric automorphisms of this algebra):
            \begin{equation}
                \delta_{a,b,c,d,e,f}\equiv\scalebox{.7}{$\begin{pmatrix}
                    a&b&0&0&0\\
                    c&-a&0&0&0\\
                    d&e&0&0&0\\
                    f&0&e&a&b\\
                    0&f&-d&c&-a
                \end{pmatrix}$}.
            \end{equation}    

            \noindent We also note that $\inner(\mfr_5)$ are just the derivations $\delta_{0,0,0,d,e,f}$. Hence, up to isometric isomorphism we may assume $d=e=f=0$ by Proposition~\ref{prop_iso_iso_reduction_one}. Moreover, there exist regular matrices $P$ with $\det P=\pm1$ and such that 
            \[
                P^{-1}\scalebox{.7}{$\begin{pmatrix}
                    a&b\\ c&-a
                \end{pmatrix}$}P=J\in\left\{\scalebox{.7}{$\begin{pmatrix}
                    0&0\\1&0
                \end{pmatrix}$};\scalebox{.7}{$\begin{pmatrix}
                    \lambda&0\\0&-\lambda
                \end{pmatrix}$}, \lambda\in\mbF^*;\scalebox{.7}{$\begin{pmatrix}
                    0&1\\\lambda&0
                \end{pmatrix}$}, \lambda\not\in\mbF^2 \right\}.
            \]
            The matrices $P$ let us define  isometric automorphisms $F$ of $(\mfr_{5},\varphi_5^\gamma)$ such that
            \[
                F\equiv\scalebox{.7}{$\begin{pmatrix}
                    P&0&0\\
                    0&\det(P)&0\\
                    0&0&\det(P)P
                \end{pmatrix}$}\quad F^{-1}\circ\delta_{a,b,c,0,0,0}\circ F\equiv\scalebox{.7}{$\begin{pmatrix}
                    J&0&0\\
                    0&0&0\\
                    0&0&J
                \end{pmatrix}$}.
            \]
            In the second case for $J$, rescaling the derivation allows us to assume that $\lambda=1$. Therefore, the quadratic algebras as extension of $(\mfr_{5},\varphi_5^\gamma)$ are:
    
            \begin{itemize}[leftmargin=*]
                \item $\bm{(\mfr_{7,4},\varphi_{7,4}^\gamma):}$ From $\delta_{0,0,1,0,0,0}$. One-parameter family, $\gamma \in \mbF^*$. The algebra has a basis $\{a,x_1,x_2,y,w_1,w_2,z\}$ with nonzero products $[a,x_1]=x_2$, $[a,w_1]=w_2$, $ [x_1,x_2]=y$, $[x_1,y]=w_1,[x_1,w_1]=z,[x_2,y]=w_2$ and $\gamma (-B_7(\pm))$ as matrix of $\varphi_{7,4}^\gamma$.
                \item $\bm{(\mfr_{7,5},\varphi_{7,5}^\gamma):}$ From $\delta_{1,0,0,0,0,0}$. One-parameter family, $\gamma \in \mbF^*$. It has a basis $\{a,x_1,x_2,y,w_1,w_2,z\}$ with nonzero products $[a,x_1]=x_1, [a,x_2]=-x_2,[a,w_1]=w_1$, $[a,w_2]=-w_2, [x_1,x_2]=y, [x_1,y]=w_1$, $[x_2,y]=w_2,[x_1,w_2]=-z, [x_2,w_1]=-z$ and $\gamma (-B_7(\pm))$ as matrix of $\varphi_{7,5}^\gamma$.
                \item $\bm{(\mfr_{7,6}^\lambda,\varphi_{7,6}^\gamma):}$ From $\delta_{0,1,\lambda,0,0,0}$. Two-parameter family, $\lambda \notin \mbF^2$, $\gamma \in \mbF^*$. It has a basis $\{a,x_1,x_2,y,w_1,w_2,z\}$ and nonzero products $[a,x_1]=\lambda x_2,[a,x_2]=x_1, [a,w_1]=\lambda w_2$, $[a,w_2]=w_1, [x_1,x_2]=y,[x_1,y]=w_1$, $[x_2,y]=w_2$, $[x_1,w_1]=\lambda z,[x_2,w_2]=-z$ and $\gamma (-B_7(\pm))$ as matrix of $\varphi_{7,6}^\gamma$.
            \end{itemize}
            It remains to consider extensions from $\mfr=\mfr_{4,1}\oplus\mbF\cdot c$ and $\mfr=\mfr_{4,2}^\lambda\oplus\mbF\cdot c$. In both cases, $\mfr=I\oplus J$, $I$ and $J$ regular and orthogonal ideals, and the algebra of skew-derivations splits into the following direct sum:
            \begin{equation}
                \der_\varphi(\mfr)=\der_{\varphi\vert_I}(I)\oplus\der_{\varphi\vert_J}(J)\oplus\{\delta\mid \delta(I)\subseteq J,\, \delta(J)\subseteq I\}.
            \end{equation}

            \noindent By Lemma \ref{lemma_der_inner},  $\der_{\varphi\vert_I}(I)=\inner(I,\varphi\vert_I)$ if $I=\mfr_{4,1}, \mfr_{4,2}^\lambda$ and $\der_{\varphi\vert_{\mbF\cdot c}}(\mbF\cdot c)=0$. Let $\{a,x,y,z\}$ the basis of $I$. By indecomposibility, we focus on those double extensions through $\delta\in \der_\varphi (\mfr)$ that interchange the summands. Since centers and derived algebras are $\delta$-invariant, $\delta(\mfr^2)\subseteq (\mbF\cdot c)^2=0$ and we assume $\delta(a)=\alpha c$, $\delta(c)=\beta z$. Denote $\nu=B(c,c)$. Then for $\mfr_{4,1}$ we have
            \[
                \beta=\varphi_{4,1}(a,\delta(c))=-\varphi_{4,1}(\delta(a),c)=-\alpha\nu.
            \]
            Rescaling $\delta$ we may assume $\alpha=1$, yielding the following algebra:
            \begin{itemize}[leftmargin=*]
                \item $\bm{(\mfr_{7,7},\varphi_{7,7}^\gamma):}$ One-parameter family, $\gamma \in \mbF^*$. The algebra has a basis $\{b,a,x,y,$ $z,c,w\}$ with nonzero products $[b,a]=c,[b,c]=-z, [a,x]=x$, $[a,y]=-y, [x,y]=z, [a,c]=w$ and \scalebox{0.7}{$\gamma\widehat{\begin{pmatrix}B_4&\bm{0}\\\bm{0}&1\end{pmatrix}}$} as matrix of $\varphi_{7,7}^\gamma$.
            \end{itemize}
            And for $(\mfr_{4,2}^\lambda,\varphi_{4,2}^{\lambda,\gamma})$ we have $\gamma\beta=\varphi_{4,2}^{\lambda,\gamma}(a,\delta(c))=-\varphi_{4,2}^{\lambda,\gamma}(\delta(a),c)=-\alpha\nu.$
            Rescaling again we may assume $\alpha=1$, arising to the algebra:
            \begin{itemize}[leftmargin=*]
                \item $\bm{(\mfr_{7,8}^{\lambda,\gamma},\varphi_{7,8}^{\lambda,\gamma,\nu}):}$ Three-parameter family, $\lambda,\gamma,\nu \in \mbF^*$. The algebra has a basis $\{b,a,x,y,z,c,w\}$ with non-zero products $[b,a]=c,\, [b,c]=-z,\, [a,x]=\lambda y$, $[a,y]=x,\, [x,y]=\lambda\gamma z,\, [a,c]=w$ and \scalebox{0.7}{$\displaystyle\nu\widehat{\begin{pmatrix}\widehat{\gamma A_\lambda}&\bm{0}\\\bm{0}&1\end{pmatrix}}$} as matrix of $\varphi_{7,8}^{\lambda,\gamma,\nu}$.
            \end{itemize}

            \noindent In dimension $7$, the following isometric isomorphisms are easily checked:
            \[
                (\mfr_{7,1},\varphi_{7,1})\cong (\mfr_{7,4},\varphi_{7,4}^\gamma), (\mfr_{7,2},\varphi_{7,2}^\gamma)\cong (\mfr_{7,7},\varphi_{7,7}^\gamma), (\mfr_{7,3}^{\lambda,\gamma},\varphi_{7,3}^{\lambda,\gamma,\nu})\cong (\mfr_{7,8}^{\lambda,\gamma},\varphi_{7,8}^{\lambda,\gamma,\nu})
            \]
            For the remaining families, non-isomorphism follows from Proposition~\ref{prop_necessary_isomorphism}. The only nilpotent algebra in the list is $\mfr_{7,1}$. The algebras $\mfr=\mfr_{7,2},\mfr_{7,3}^{\lambda,\gamma}$ have $2$-dimensional center, but for a given $w\in \mfr\setminus \mfr^2$, the elementary divisors of the map $\overline{\ad w}$ on the (abelian) algebra $\mfr^2/Z(\mfr)$ are $x$, $(x-k)$, $(x+k)$ if $\mfr=\mfr_{7,2}$ and $x, (x^2-\lambda k^2)$ if $\mfr=\mfr_{7,3}^{\lambda,\gamma}$, so they are not isomorphic. In the case of $\mfr=\mfr_{7,5},\mfr_{7,6}^\lambda$, both algebras have a one-dimensional center and satisfy $\mfr^2/Z(\mfr)\cong \mfr_5$. 
            Let $f=k\delta+\ad_{\mfr_5}\, x$ be where $\delta$ is the outer derivation defining the double extension $\mfr=(\mfr_5)_\delta$ namely $\mfr_{7,5}$ or $\mfr_{7,6}^\lambda$. It is readily checked that the elementary divisors of $f\vert_{Z(\mfr_5)}$ are $(x-k),(x+k)$ when $\mfr=\mfr_{7,5}$, whereas it is $(x^2-\lambda k^2)$ if $\mfr=\mfr_{7,6}$. Hence these algebras are not isomorphic.

            The work along this section is summarized in the following result:

            \begin{theorem}
                Up to dimension $7$, the indecomposable and solvable quadratic Lie algebras over fields of characteristic not two are those included in Table~\ref{tab_indecomposable_up_to_seven}. 
            \end{theorem}

            \begin{table}[ht]
                \addtolength{\tabcolsep}{-0.3em}
                \scriptsize
                \centering
                \begin{tabularx}{\textwidth}{>{\hsize=.2\hsize}X>{\hsize=.6\hsize}X>{\hsize=.7\hsize}XX}\toprule
                    $\bm{\mfr}$ & Brackets & & Invariant bilinear form \\
                    \midrule
                    $\bm{\mfr_1}$ & $[x,x]=0$ & & $\gamma X^2$\\[1mm]
                    \hdashline[1pt/1pt]
                    &&&\\[-2mm]
                    $\bm{\mfr_{4,1}}$ & $[a,x]=x$ & $[a,y]=-y,$  & $2AZ+2X_1X_2$\\
                    &$[x,y]=z$&&\\ 
                    $\bm{\mfr^\lambda_{4,2}}$ & $[a,x]=\lambda y$ &$[a,y]=x$ & $2AZ-\lambda X_1^2+X_2^2$\\
                    &$[x,y]=\lambda z$ & &\\[1mm]
                    \hdashline[1pt/1pt]
                    &&&\\[-2mm]
                    $\bm{\mfr_5}$ & $[x_1,x_2]=y$ &$[x_1,y]=z_1$ & $\gamma(2X_1Z_2-2X_2Z_1+Y^2)$\\
                    &$[x_2,y]=z_2$&&\\[1mm]
                    \hdashline[1pt/1pt]
                    &&&\\[-2mm]
                    $\bm{\mfr_{6,1}}$& $[a,x_1]=x_2$ &$[a,y_1]=-y_2$ &$2AZ+2X_1Y_2+2X_2Y_1$\\
                    &$[x_1,y_1]=z$&&\\
                    $\bm{\mfr^\lambda_{6,2}}$& $[a,x_1]=x_1$ &$[a,x_2]=-x_2$ & $2AZ+2X_1X_2+2Y_1Y_2$\\
                    &$[a,y_1]=\lambda y_1$ &$[a,y_2]=-\lambda y_2$&\\
                    &$[x_1,x_2]=z$&$[y_1,y_2]=\lambda z$&\\
                    $\bm{\mfr_{6,3}}$ & $[a,x_1]=x_1+x_2$ &$[a,x_2]=x_2$ & $2AZ+2X_1Y_2+2X_2Y_1$\\
                    &$[a,y_1]=-y_1-y_2$&$[a,y_2]=-y_2$&\\
                    &\multicolumn{2}{l}{$[x_1,y_1]=[x_1,y_2]=[x_2,y_1]=z$}&\\
                    $\bm{\mfr^\lambda_{6,4}}$ & $[a,x_1]=x_1$ &$[a,x_2]=-x_2$ & $\gamma(2AZ+2X_1X_2-\lambda Y_1^2+Y_2^2)$\\
                    &$[a,y_1]=\lambda y_2$&$[a,y_2]=y_1$&\\
                    &$[x_1,x_2]=z$&$[y_1,y_2]=\lambda z$&\\
                    $\bm{\mfr^{\lambda,\mu,\gamma}_{6,5}}$ & $[a,x_1]=\lambda x_2$ &$[a,x_2]=x_1$ &$\nu(2AZ-\lambda X_1^2+X_2^2-\gamma\mu Y_1^2+\gamma Y_2^2)$\\
                    &$[a,y_1]=\mu y_2$&$[a,y_2]=y_1$&\\
                    &$[x_1,x_2]=\lambda z$&$[y_1,y_2]=\gamma\mu z$&\\
                    $\bm{\mfr^{\lambda}_{6,6}}$ & $[a,x_1]=\lambda y_1+x_2$ &$[a,y_1]=x_1-y_2$ & $\gamma(2AZ+2X_1Y_2+2X_2Y_1)$\\
                    &$[a,x_2]=-\lambda y_2$&$[a,y_2]=-x_2$&\\
                    &$[x_1,y_1]=[y_1,y_2]=z$&$[x_1,x_2]=\lambda z$&\\
                    $\bm{\mfr^{\lambda}_{6,7}}$ & $[a,x_1]=y_1$ &$[a,y_1]=-\lambda x_1+2y_1$ & $2AZ+2X_1Y_2+2X_2Y_1$\\
                    &$[a,x_2]=-2x_2-y_2$&$[a,y_2]=\lambda x_2$&\\
                    &$[x_1,x_2]=z$&$[y_1,x_2]=2z$&\\
                    &$[y_1,y_2]=-\lambda z$&&\\
                    $\bm{\mfr^{\lambda,\mu,\gamma}_{6,8}}$ & $[a,x_i]=x_{i+1},\, i\leq 3$ &$[a,x_4]=-\mu x_1+\lambda x_3$ & $\nu(2AZ+\gamma X_1^2-2X_1X_3+X_2^2$\\
                    &$[x_1,x_2]=z$&$[x_1,x_4]=(\gamma\mu+\lambda)z$&$+2(\gamma\mu+\lambda)X_2X_4-(\gamma\mu+\lambda)X_3^2$\\
                    &$[x_2,x_3]=-(\gamma\mu+\lambda)z$&$[x_3,x_4]=(\gamma\lambda\mu+\lambda^2-\mu)z$&$+(\gamma\lambda\mu+\lambda^2-\mu)X_4^2)$\\[1mm]
                    \hdashline[1pt/1pt]
                    &&&\\[-2mm]
                    $\bm{\mfr_{7,1}}$ & $[a,x_i]=x_{i+1},\, i\leq 4$ &$[x_1,x_4]=z$ &$\gamma(2AZ-2X_1X_5+2X_2X_4-X_3^2)$\\
                    &$[x_2,x_3]=-z$&&\\
                    $\bm{\mfr_{7,2}}$& $[a,x_1]=x_1$ &$[a,x_2]=-x_2$ & $\gamma(2AZ+2X_1X_2-2Y_1Y_3+Y_2^2)$\\
                    &$[a,y_i]=y_{i+1},\, i\leq 2$&$[x_1,x_2]=[y_1,y_2]=z$&\\
                    $\bm{\mfr^{\lambda,\gamma}_{7,3}}$ & $[a,x_1]=\lambda x_2$ &$[a,x_2]=x_1$ & $\nu(2AZ-\gamma\lambda X_1^2+\gamma X_2^2-2Y_1Y_3+Y_2^2)$\\
                    &$[a,y_i]=y_{i+1},\, i\leq 2$&$[x_1,x_2]=\lambda\gamma z$&\\
                    &$[y_1,y_2]=z$&\\
                    $\bm{\mfr_{7,5}}$ & $[a,x_1]=x_1$ &$[a,x_2]=-x_2$ &$\gamma(-2AZ+2X_1W_2-2X_2W_1+Y^2)$\\
                    &$[a,w_1]=w_1$&$[a,w_2]=-w_2$&\\
                    &$[x_1,x_2]=y$&$[x_1,y]=w_1$&\\
                    &$[x_2,y]=w_2$&$[x_1,w_2]=[x_2,w_1]=-z$&\\
                    $\bm{\mfr^\lambda_{7,6}}$ & $[a,x_1]=\lambda x_2$ &$[a,x_2]=x_1$ &$\gamma(-2AZ+2X_1W_2-2X_2W_1+Y^2)$\\
                    &$[a,w_1]=\lambda w_2$&$[a,w_2]=w_1$&\\
                    &$[x_1,x_2]=y$&$[x_1,y]=w_1$&\\
                    &$[x_2,y]=w_2$&$[x_1,w_1]=\lambda z$&\\
                    &$[x_2,w_2]=-z$&&\\[1mm]
                    \bottomrule
                \end{tabularx}
                \caption{Indecomposable Lie algebras up to dimension $7$ in characteristic different from $2$.}
                \label{tab_indecomposable_up_to_seven}
            \end{table}   
    \section{Algebraically and real closed fields}\label{section_alg_real_closed_fields}
        The isomorphism and isometric isomorphism problems for a given $n$-parameter family of algebras in the previous section are beyond the scope of the present work. We address these questions only over algebraically closed and real closed fields. This allows us to compare our results with existing classifications over the complex and real fields, such as those in \cite{campoamor2008quasi} (up to dimension 6 over the reals) and the preprint  \cite{duong2014solvable} (up to dimension $8$ over the complex field). The first and simplest step is to rescale the bilinear forms. 
        \begin{lemma}\label{lemma_rescaling_squares}
            Let $(V, \varphi)$ be an orthogonal vector space, $\delta\in \der_\varphi (V)$ and $(V_\delta,\varphi_\delta)=\mfd(V, \varphi,\delta)$. Then, $(V_\delta,\varphi_\delta)$ and $(V_\delta,\gamma\varphi_\delta)$ are isometrically isomorphic for any $0\neq\gamma\in\mbF^2$.
        \end{lemma}
        \begin{proof}
            For a standard basis of $V_\delta$, $\{\delta, v_1, \dots, v_n, z\}$, define $f(\delta)=\delta$, $f(z)=\gamma z$ and $f(v_i)=\sqrt{\gamma}v_i$. Then, extending $f$ linearly we obtain an isometric isomorphism $f\colon (V_\delta,\gamma\varphi_\delta)\to (V_\delta,\varphi_\delta)$.
        \end{proof}
        
        \subsection{Algebraically Closed Fields}
            For these fields, every irreducible polynomial is linear. Hence, the only elementary divisors are of the form $(x-\lambda)^k$, $k\geq 1$ and $\lambda \in \mbF$. Moreover, by Lemma~\ref{lemma_rescaling_squares}, the bilinear form $\varphi$ can be normalized to $1$.

            \begin{theorem}\label{indecomposable_algebraically_closed}
                Up to isometric isomorphisms,  the indecomposable and solvable quadratic Lie algebras up to dimension $7$ over algebraically closed fields of characteristic not two are in the following list:
                \begin{enumerate}
                    \item The one-dimensional algebra $(\mfr_1,\varphi_1^1)$.
                    \item The $4$-dimensional algebra  $(\mfr_{4,1},\varphi_{4,1})$.
                    \item The $5$ and $6$-dimensional nilpotent algebras  $(\mfr_5,\varphi_5^1)$ and $(\mfr_{6,1},\varphi_{6,1})$.
                    \item The $6$-dimensional algebras $(\mfr_{6,2}^\lambda,\varphi_{6,2})$ for any $\lambda \neq 0$, and  $(\mfr_{6,3},\varphi_{6,3})$.
                    \item The $7$-dimensionals algebras $(\mfr_{7,1},\varphi_{7,1}^1)$, $(\mfr_{7,2},\varphi_{7,2}^1)$ and $(\mfr_{7,5},\varphi_{7,5}^ 1)$.
                \end{enumerate} 
                The algebras in the list are pairwise non-isometrically isomorphic, except for the one-parameter family $\mfr_{6,2}^\lambda$. Two algebras $\mfr_{6,2}^\lambda$ and $\mfr_{6,2}^\mu$ are isometrically isomorphic if and only if $\mu=\pm \lambda^{\pm1}$.     
            \end{theorem}
            \begin{proof}
                The result follows from Section~\ref{section_constructions}. It remains to prove when distinct members of the parametric family in item (4) are not isometrically isomorphic. Note that $\overline{\ad\,a}$ on $\mfg:=(\mfr_{6,2}^\lambda)^2/Z(\mfr_{6,2}^\lambda)$ has (possibly) $4$ distinct eigenvalues, $\pm1$ and $\pm\lambda$. And for a given $w\in \mfr_{6,2}^\lambda\setminus (\mfr_{6,2}^\lambda)^2$, the elementary divisors of $\overline{\ad w} $ on $\mfg$  are $\pm k, \pm k\lambda$. If $\mfr_{6,2}^\lambda\cong \mfr_{6,2}^\mu$ there exists $w\in \mfr_{6,2}^\mu$ with the same eigenvalue as $\overline{\ad\,a}$, so $k\mu,k\in\{\pm1,\pm\lambda\}$. If $\pm k=1$ then $\mu=\pm\lambda$ and $k\mu=\pm1$ if $\pm k=\lambda$ and therefore $\mu=\pm\lambda^{-1}$. That is, $\mu\in\{\pm\lambda^{\pm1}\}$. A quick check shows that these four cases provide  isomorphic algebras.
            \end{proof}
    
            \begin{remark}\label{remark_isomorphism_6_2}
                It should be noted that the isometric isomorphism criterion $(\mfr_{6,2}^\lambda, \varphi_{6,2})\cong (\mfr_{6,2}^\mu,\varphi_{6,2})$ iff  $\mu\in\{\pm\lambda^{\pm1}\}$ is also valid over arbitrary fields.
            \end{remark}
            \begin{remark}
                By Theorem~\ref{indecomposable_algebraically_closed}, the classification of indecomposable quadratic Lie algebras of dimension $8$ over algebraically closed fields is tractable. Starting from a $6$-dimensional quadratic algebra, the main computational challenge is determininig its outer skew derivations 
            \end{remark}
   
        \subsection{Real Closed Fields}
            Over such fields, every irreducible polynomial has degree one or two. Moreover, an element $\lambda\in\mbF$ is a square if and only if $\lambda\geq 0$. In that case, there exists a unique $\mu\geq 0$ such that $\mu^2=\lambda$. Hence,  $x^2-\lambda$ is irreducible if and only if $\lambda<0$. By Lemma~\ref{lemma_rescaling_squares}, some of the bilinear forms $\gamma\varphi$ can be rescaled to $\pm1$.

            The above remarks reduce the parametric families of dimension one and five to exactly four pairwise non-isometrically isomorphic algebras, namely $(\mfr_1, \pm\varphi_1)$ and $(\mfr_5, \pm\varphi_5)$. In dimension $4$, every irreducible elementary divisor $x^2-\lambda$ can be reduced to $x^2+1$ by rescaling $\delta\in\mfso(V,\varphi)$ by $\sqrt{-\lambda}$. (see Table~\ref{table_divisors_structure}). Therefore, we obtain three non isometrically isomorphic algebras: $(\mfr_{4,1},\varphi_{4,1})$ and $(\mfr_{4,2}^{-1},\varphi_{4,2}^{-1,1}), (\mfr_{4,2}^{-1},-\varphi_{4,2}^{-1,-1})$, the latter two having signatures $(3,1)$ and $(1,3)$. 

            In the six and seven dimensional cases, some of the matrices in Table~\ref{table_matrices} used to describe the corresponding quadratic algebras can be further refined.

            \begin{table}[h]
                \scriptsize
                \centering
                \begin{tabularx}{\textwidth}{>{\hsize=.2\hsize}XX>{\centering\arraybackslash\hsize=.5\hsize}X>{\centering\arraybackslash\hsize=.5\hsize}X}
                    \toprule
                    $V$ &\bf{Elementary Divisors} & $M(\delta)$ & $M(\varphi)$\\
                    \midrule
                    \multirow{4}{*}{$V_1$}&$(x^2+\lambda^2), \lambda>0$ & \scalebox{.8}{$\begin{pmatrix}
                        0&-\lambda\\\lambda&0
                    \end{pmatrix}$ }&\scalebox{.8}{$\pm\begin{pmatrix}
                        1&0\\0&1
                    \end{pmatrix}$}\\[2mm]
                    &&\\[-2mm]
                    &$(x^2+\lambda^2)^2, \lambda>0$ & \scalebox{.8}{$\begin{pmatrix}
                        0&-\lambda&0&0\\\lambda&0&0&0\\1&0&0&\lambda\\0&-1&-\lambda&0
                    \end{pmatrix}$} &$\pm B_4$\\
                    &&\\[-2mm]
                    \hdashline[1pt/1pt]
                    &&\\[-2mm]
                    $V_2$&$\displaystyle(x^2-2x+(1+\lambda^2)), \newline(x^2+2x+(1+\lambda^2)),\,\lambda> 0$ & \scalebox{.8}{$\begin{pmatrix}
                        1&\lambda&0&0\\-\lambda&1&0&0\\0&0&-1&-\lambda\\0&0&\lambda&-1
                    \end{pmatrix}$} &$B_4$\\
                    \bottomrule
                \end{tabularx}
                \caption{Matrices corresponding to nonlinear elementary divisors over real closed fields.}
                \label{table_matrices_real_closed}
            \end{table}

            By rescaling and using Tables~\ref{table_matrices} and~\ref{table_matrices_real_closed}, we obtain simpler descriptions of the quadratic algebras $(\mfr_{6,i}^*, \varphi_{6,i}^*)$, $i=4,5,6,7$ and $(\mfr_{7,j}^*, \varphi_{7,j}^*)$, $j=3,6$ over real closed fields.

            \begin{itemize}[leftmargin=*]
                \item $\bm{((\mfr_{6,4}^\lambda)^\dagger,\pm(\varphi_{6,4})^\dagger):}$ Elementary divisors $(x-1), (x+1)$ and $(x^2+\lambda^2), \lambda>0$. The algebra $(\mfr_{6,4}^\lambda)^\dagger$ has a basis $\{a,x_1,x_2,$ $y_1,y_2,z\}$ whose nonzero products and bilinear form are adjusted accordingly $[a,x_1]=x_1$, $[a,x_2]=-x_2$, $[a,y_1]=\lambda y_2$, $[a,y_2]=-\lambda y_1$, $[x_1,x_2]=z$, $[y_1,y_2]=\lambda z$ and \scalebox{0.8}{$\widehat{\begin{pmatrix}B_2&\bm{0}\\\bm{0}&I_2\end{pmatrix}}$}, $I_2=$\scalebox{0.7}{$\begin{pmatrix}  1&0\\0&1\end{pmatrix}$}, as matrix of $\varphi_{6,4}^\dagger$.
                \item $\bm{((\mfr_{6,5}^\lambda)^\dagger(\epsilon),\pm(\varphi_{6,5})^\dagger(\epsilon)), \epsilon=\pm 1:}$ Elementary divisors $(x^2+1), (x^2+\lambda^2)$ $\lambda\geq 1$. Four one-parameter families with parameter  $\lambda\geq1$. The algebra $(\mfr_{6,5}^\lambda)^\dagger(\epsilon)$ has a basis $\{a,x_1,x_2, y_1,y_2,$ $z\}$ whose nonzero products and bilinear form are adjusted accordingly $[a,x_1]=x_2$, $[a,x_2]=-x_1$, $[a,y_1]=\lambda y_2$, $[a,y_2]=-\lambda y_1$, $[x_1,x_2]=z$, $[y_1,y_2]=\epsilon\lambda z$ and \scalebox{0.7}{$\widehat{\begin{pmatrix}I_2&\bm{0}\\\bm{0}&\epsilon I_2\end{pmatrix}}$} as matrix of $(\varphi_{6,5})^\dagger(\epsilon)$.
                \item $\bm{((\mfr_{6,6})^\dagger,\pm(\varphi_{6,6})^\dagger):}$ Elementary divisor $(x^2+1)^2$. The algebra $(\mfr_{6,6})^\dagger$ has a basis $\{a,x_1,y_1,x_2,y_2,z\}$ with nonzero products $[a,x_1]=y_1+x_2, [a,y_1]=-x_1-y_2, [a,x_2]=- y_2$, $[a,y_2]=x_2$, $[x_1,x_2]=[x_1,y_1]=- [y_1,y_2]=z$ and $ B_6$ as matrix of $(\varphi_{6,6})^\dagger$.
                \item $\bm{((\mfr_{6,7}^\lambda)^\dagger,(\varphi_{6,7})^\dagger):}$ Elementary divisors $(x^2-2x+(1+\lambda^2)),\, (x^2+2x+(1+\lambda^2)),\, \lambda >0$. One-Parameter family with parameter $\lambda >0$. The algebra $(\mfr_{6,7}^\lambda)^\dagger$ has a basis $\{a,x_1,y_1,$ $x_2,y_2,z\}$ whose nonzero products and bilinear form are adjusted accordingly $[a,x_1]=x_1-\lambda y_1, [a,y_1]=\lambda x_1+y_1$, $[a,x_2]=-x_2+\lambda y_2$, $[a,y_2]=-\lambda x_2-y_2, [x_1,x_2]=-\lambda z, [x_1,y_2]=[y_1,x_2]=z, [y_1,y_2]=\lambda z$ and $B_6$ as matrix of $(\varphi_{6,7})^\dagger$.
                \item $\bm{((\mfr_{7,3})^\dagger(\epsilon),\pm (\varphi_{7,3})^\dagger(\epsilon)):}$ Elementary divisors $(x^2+1)$ and $x^3$. One-parameter family with $\lambda\not\in\mbF^2, \gamma,\nu\in \mbF^*$ The algebra $(\mfr_{7,3})^\dagger(\epsilon)$ has a basis $\{a,x_1,x_2,x_3,y_1,y_2,z\}$ whose nonzero products and bilinear form are adjusted accordingly $[a,x_i]=x_{i+1}, i=1,2$, $[a,y_1]=y_2$, $[a,y_2]=-y_1$,  $[x_1,x_2]= z$, $[y_1,y_2]=\epsilon z$ and \scalebox{0.7}{$\widehat{\begin{pmatrix}-B_3(\pm) &\bm{0}\\\bm{0}&\epsilon I_2\end{pmatrix}}$} as matrix of $(\varphi_{7,3})^\dagger(\epsilon)$.
                \item \bm{$((\mfr_{7,6})^\dagger,\pm (\varphi_{7,6})^\dagger):$} Rescaling the outer derivation of $\mfr_5$, $\delta_{0,1,\lambda,0,0,0}$,  we can assume $\lambda=-1$. The algebra $(\mfr_{7,6})^\dagger$ has a basis $\{a,x_1,x_2,y,w_1,w_2,z\}$ whose nonzero products and bilinear form are adjusted accordingly $[a,x_1]=x_2$, $[a,x_2]=-x_1$, $[a,w_1]=w_2$, $[a,w_2]=-w_1$, $[x_1,x_2]=y$, $[x_1,y]=w_1$, $[x_2,y]=w_2$, $[x_1,w_1]=-z$, $[x_2,w_2]=-z$ and $-B_7(\pm)$ as matrix of $(\varphi_{7,6})^\dagger$.        
            \end{itemize}

            \noindent The foregoing discussion is summarized in our final result.
            \begin{theorem}\label{indecomposable_real_closed_fields}
                Up to dimension $7$, the indecomposable and solvable quadratic Lie algebras over real closed fields of characteristic not two can be obtained from those listed in Table~\ref{tab_indecomposable_up_to_seven} by rescaling the irreducible polynomials. Up to isometric isomorphisms the complete list is as follow:
                \begin{enumerate}
                    \item The one-dimensional algebras $(\mfr_1,\pm\varphi_1^1)$.
                    \item The $4$-dimensional algebras $(\mfr_{4,1},\varphi_{4,1})$,    $(\mfr_{4,2}^{-1},\pm\varphi_{4,2}^{-1,1})$.
                    \item The $5$ and $6$-dimensional nilpotent algebras  $(\mfr_5,\pm\varphi_5^{1})$ and $(\mfr_{6,1},\varphi_{6,1})$.
                    \item The $6$-dimensional algebras $(\mfr_{6,2}^\lambda,\varphi_{6,2})$ with $ \lambda \geq 1$, $(\mfr_{6,3},\varphi_{6,3})$,  
                    \item[] $((\mfr_{6,4}^\lambda)^\dagger,\pm (\varphi_{6,4})^{\dagger})$ with $\lambda>0$, $((\mfr_{6,5}^\lambda)^\dagger (\epsilon),\pm(\varphi_{6,5})^\dagger(\epsilon))$ with $\lambda \geq 1$, 
                    $((\mfr_{6,6})^\dagger,\pm(\varphi_{6,6})^\dagger)$ and $((\mfr_{6,7}^\lambda)^\dagger,(\varphi_{6,7})^\dagger)$.
                    \item The $7$-dimensional algebras $(\mfr_{7,1},\pm\varphi_{7,1}^1)$, $(\mfr_{7,2},\pm\varphi_{7,2}^1)$, $(\mfr_{7,3}^\dagger(\epsilon),\pm\varphi_{7,3}^\dagger(\epsilon))$, 
                    \item[]$(\mfr_{7,5},\pm\varphi_{7,5}^1)$ and $((\mfr_{7,6})^\dagger,\pm(\varphi_{7,6})^\dagger)$.
                \end{enumerate}
                The algebras in the list are pairwise non-isometrically isomorphic.
            \end{theorem}
            \begin{proof}
                Two families in different items are non-isomorphic just as has been proven in the previous section. By  Remark~\ref{remark_isomorphism_6_2}, $\mfr^\lambda_{6,2}\cong \mfr^\mu_{6,2}$ with $\lambda,\mu\geq 1$ if and only if $\lambda=\mu$. For the rest of one-parameter families, an analogue result is also true and can be proven using the elemental divisors of some generic elements on the quotient $\mfr^2/Z(\mfr)$. 

                Next we see $(\mfr^\lambda_{6,5})^\dagger(1)$ and $(\mfr_{6,5}^\lambda)^\dagger(-1)$ are not isomorphic. Let $\mfr=(\mfr^\lambda_{6,5})^\dagger(\epsilon)$ be, and notice that $[\mfr^2,\mfr^2]=Z(\mfr)=\mbF\cdot z$. Define a symmetric bilinear form on $\mfr^2/Z(\mfr)$ by $\overline{B_w}(\bar{x},\bar{y})=B_w(x,y)$, where $[\ad w(x),y]=B_w(x,y)z$ and $x,y\in \mfr^2$, $w\in \mfr$. If $w=\gamma a+z$ with $z\in\mfr^2$ then $[\ad w(x),y]=-\gamma\varphi(\delta(x),\delta(y))\delta^*$, as $z=\pm\delta^*$, if $\gamma\neq 0$, the bilinear form $B_w$ has, up to sign change, the same definiteness as $\varphi(\delta(\cdot),\delta(\cdot))$. In the algebra $(\mfr_{6,5}^\lambda)^\dagger(1)$, $\overline{B_w}$ is zero or definite, in contrast, for $(\mfr^\lambda_{6,5})^\dagger(-1)$ it is zero or indefinite. If there is an isomorphism $f\colon (\mfr_{6,5}^\lambda)^\dagger(1)\cong (\mfr_{6,5}^\lambda)^\dagger(-1)$, $[\ad\, f(w)(f(x)),f(y)]=[\ad\, w(x),y]$ for all $x,y\in (\mfr_{6,5}^\lambda)^\dagger(-1)$ and thus
                \[
                    \overline{B_{f(w)}}(\overline{f(x)},\overline{f(y)})=\overline{B_{f(w)}}(\bar{f}(x),\bar{f}(y))=\overline{B_w}(\bar{x},\bar{y}).
                \]
                A contradiction by the different definiteness of the bilinear forms.

                Finally, we prove that $\mfr_{7,3}(1)$ and $ \mfr_{7,3}(-1)$ are not isomorphic. Denoting $\mfr=\mfr_{7,3}(\epsilon)$ we also have $[\mfr^2,\mfr^2]=\mbF\cdot z\subseteq Z(\mfr)$, is an invariant subspace for any isomorphism. Consider the ideal $I=z^\perp=\mfr^2\oplus\mbF\cdot x_1$. The derived algebra is preserved by isomorphisms and as $[x_1,\mfr^2]=\mbF\cdot z$, $B(f(x_1),f(z))=B(f(x_1),[f(x_1),f(y)])=B([f(x_1),f(x_1)],f(y))=0$ and thus $z^\perp$ is also preserved by isomorphisms. Consider now the bilinear form in $z^\perp / \mbF\cdot z$ given by $\overline{B_w}(\bar{x},\bar{y})=B_w(x,y)$ and $[\ad\, w(x),y]=B(x,y)z$, here $w\in \mfr$ and $x,y\in z^\perp$. For the algebra $\mfr_{7,3}(1)$ this bilinear form is zero or semidefinite, in contrast in $\mfr_{7,3}(-1)$ it is indefinite or zero. Reasoning as in the previous paragraph, the algebras can not be isomorphic.
            \end{proof}
    \section{Epilogue}    
        The existing literature of solvable quadratic Lie algebras contains only a limited number of classification results. We summarize these results below and discuss their scope and limitations.
        \begin{enumerate}
            \small{
                \item [1987:] \cite{Favre_Santharoubane_1987}, quadratic nilpotent up to dimension $7$ in characteristic not 2. The algebras  $\mfr_1$, $n_{2,3}$, $n_{3,2}$ and $\mfr_{7,1}$ exhaust the indecomposable list.
                \item [2007:] \cite{Kath_2007} real nilpotents up to dimension $10$.
                \item [2008:] \cite{campoamor2008quasi}, real solvable up to dimension $6$. The indecomposable algebras are covered by items (1) to (4) in Theorem \ref{indecomposable_real_closed_fields}.
                \item [2014:] \cite{duong2014solvable}, complex solvable up to dimension $8$. Up to dimension $7$ are covered by Theorem \ref{indecomposable_algebraically_closed}.
                \item [2014:] \cite{Benayadi_Elduque_2014} Complex and reals nonsolvable Lie algebras up to dimension $13$. Those algebras  are double extensions of solvable quadratic Lie algebras up to dimension $7$ by a simple $3$-dimensional algebra.
                \item [2017:]\cite{Benito_delaConcepcion_Laliena_2017}, complex and real nilpotents of $2$ or $3$ generators and low nilindex.
            }
        \end{enumerate}
        The results along this paper extend the classification of solvable quadratic Lie algebras up to dimension $7$ to arbitrary fields and provide a foundation for further developments via iterated double extensions. Except for \cite{Kath_2007,Benayadi_Elduque_2014,Benito_delaConcepcion_Laliena_2017}, previous classifications rely primarily on the double extension method. As shown in Sections \ref{general_section} and \ref{section_constructions}, this approach remains tractable for solvable Lie algebras of dimension $8$ and nilpotent ones up to dimension $12$, although it becomes considerably more intricate in higher dimensions, especially over non-algebraically closed fields. Combining double extensions with other techniques developed in \cite{Kath_2007,Benito_delaConcepcion_Laliena_2017,Benayadi_Elduque_2014} may be useful for studying and classifying subclasses of quadratic Lie algebras, such as nilpotent and nonsolvable , as well as the class of Lie algebras of quadratic dimension $2$ treated in \cite{bajo2007lie}. Further algebraic and geometric approaches to quadratic Lie algebras are developed in \cite{garcia2020invariant,alvarez2021deformation} and references therein.

\end{document}